\documentclass[a4paper,12pt]{article}
\usepackage{courier}
\usepackage{blindtext} 
\usepackage{amsmath}
\usepackage{amssymb}
\usepackage[colorlinks=true,breaklinks]{hyperref}
\usepackage[hyphenbreaks]{breakurl}
\usepackage{xcolor}
\definecolor{c1}{rgb}{0,0,1}
\definecolor{c2}{rgb}{0,0.3,0.9}
\definecolor{c3}{rgb}{0.3,0.9}
\hypersetup{linkcolor={c1},citecolor={c2},urlcolor={c3}}
\usepackage{enumerate}
\usepackage{todonotes}
\usepackage{makeidx}
\makeindex
\usepackage[top=3.5cm,bottom=3.5cm,left=2.8cm,right=2.8cm]{geometry}
\usepackage{fancyhdr}

\def\XXint#1#2#3{{\setbox0=\hbox{$#1{#2#3}{\int}$ }
\vcenter{\hbox{$#2#3$ }}\kern-.6\wd0}}

\usepackage{amsthm}
\theoremstyle{plain}
\newtheorem{theorem}{Theorem}[section]

\theoremstyle{definition}

\theoremstyle{lemma}
\newtheorem{lemma}[theorem]{Lemma}

\theoremstyle{Remark}

\theoremstyle{proposition}

\theoremstyle{corollary}

\theoremstyle{example}

\theoremstyle{assumption}

\usepackage{systeme}
\usepackage{colonequals}
\usepackage{comment}
\begin{document}
\pagestyle{empty}

\title{Hypocoercivity for the non-linear semiconductor Boltzmann equation}

\author{Marlies Pirner$^\ast$, Gayrat Toshpulatov\thanks{Institut f\"ur Analysis und Numerik, Fachbereich Mathematik
und Informatik der Universit\"at M\"unster, Orl\'eans-Ring 10, 48149 M\"unster, Germany, {\tt marlies.pirner@uni-muenster.de, gayrat.toshpulatov@uni-muenster.de}}}
\maketitle

\pagestyle{plain}
\begin{abstract}
A kinetic model for \textcolor{black}{semiconductor devices} is considered on a flat torus. We prove exponential decay to equilibrium  for this non-linear kinetic model by hypocoercivity estimates. This seems to be  the first hypocoercivity result for this nonlinear kinetic equation for semiconductor devices without smallness assumptions. The analysis benefits from uniform bounds of the solution in terms of the equilibrium velocity distribution. 
\end{abstract}
\textcolor{black}{\begin{small}\textbf{Keywords:} Kinetic theory of gases, semiconductor Boltzmann equation, long time
behavior, convergence to equilibrium, hypocoercivity, Lyapunov functional.
\end{small}}
\section{Introduction}
In semi-classical kinetic description, the statistical evolution of particles (electrons or holes) in semiconductor devices can be described by 
the following spatially inhomogeneous Boltzmann equation \cite{Book1, Book2, Book3, Pop.}
\begin{equation*}
\begin{cases}
\partial_t f+\nabla_{v}E(v)\cdot \nabla_x f=Q(f)\\
f(0,x,v)=f_0(x,v).
\end{cases}
\end{equation*}
Here the variables $t\geq 0$ and  $x \in \mathbb{T}^d$   stand for time and  position, respectively. The variable $v$ stands for wave vector and it belongs to a domain $B\subset \mathbb{R}^d$ which is called \emph{Brillouin zone}. The unknown $f(t,x,v)\geq 0$ describes the phase space distribution function of particles. The operator $\partial_t+\nabla_vE(v)\cdot \nabla_x$ describes transport, while $Q$ is a collision operator given by
\begin{equation*}
Q(f)=\int_{B}\sigma(v,v')\left[M(1-f)f'-M'(1-f')f\right]dv',
\end{equation*}
where $\displaystyle M=M(v)=\frac{1}{(2\pi k_BTm/{\hbar}^2)^{d/2}}e^{-\frac{|E(v)|^2}{k_BT}},$ $M'=M(v'),$ $f'=f(t,x,v').$ The (positive) physical constants denoted by $k_B,$ $T,$ $m$ and $\hbar$ are respectively the Boltzmann constant, the temperature, the particle mass, and the Planck constant.   $\sigma(v,v')$ is a given non-negative symmetric  function (i.e. $\sigma(v,v')=\sigma(v',v)$) and the product $\sigma(v,v')M(v)$ presents  the probability for a particle to change its wave vector $v$ into
another $v'$ during an interaction. 
 The function $E(v)$ describes the energy of particles and its gradient $\nabla_v E(v)$ describes the velocity of particles.

In this paper we consider the \emph{parabolic band approximation}  \cite{Book1, Pop., Book2} which consist in taking
$$E(v)=\frac{\hbar}{2m}|v|^2,\, \, \, \, B=\mathbb{R}^d.$$
This means it is assumed that the particles under consideration have energies close to an extremum of an energy band which therefore can be approximated by a paraboloid. 
For simplicity, we set all physical constants to unity: $k_B=T=m=\hbar=1.$ This is not a restriction since one can check that our results hold without this condition. Therefore, we shall consider the normalized equation
\begin{equation}\label{Eq}
\begin{cases}
\partial_t f+v\cdot \nabla_x f=Q(f), \, \, \, t>0,\, x\in \mathbb{T}^d,\, v\in \mathbb{R}^d,\\
f(0,x,v)=f_0(x,v),
\end{cases}
\end{equation}
with \begin{equation*}
Q(f)=\int_{\mathbb{R}^d}\sigma(v,v')\left[M(1-f)f'-M'(1-f')f\right]dv', \, \, \,  M(v)=\frac{1}{(2\pi )^{d/2}}e^{-\frac{|v|^2}{2}}.
\end{equation*} 
 
Equation \eqref{Eq} has several properties following standard physical considerations. 
 The collision operator $Q$ \textit{acts only on the velocity $v,$} which reflects the physical fact that collisions are localized in space. The factors $(1-f)$ and $(1-f')$ appearing in $Q$ take into account the Pauli exclusion principle. The distribution function is supposed to satisfy the bound 
\begin{equation}\label{01}
0\leq f\leq 1.
\end{equation}

Whenever $f(t,x,v)$ is a (well-behaved) solution of \eqref{Eq}, one has  \textit{global conservation of mass} \cite[Lemma 2.1]{P.S}
\begin{equation*}
\int_{\mathbb{R}^d}f(t,x,v)dvdx=\int_{\mathbb{R}^d}f_0(x,v)dvdx, \, \, \, \, \forall\, t\geq 0.
\end{equation*}
However, there is no other conservation law.

Equation \eqref{Eq} is \emph{dissipative} in the sense that the following entropy functionals  decrease under the time-evolution of $f$ (see \cite{Neu.Schmeis}): let $\chi=\chi(z), \, \, z \in (0,1)$ be an arbitrary increasing function and $S_{\chi}=S_{\chi}(z,v)$ be a function satisfying $$ \frac{\partial S_{\chi}(z,v)}{\partial z}=\chi\left(\frac{z}{M(v)(1-z)}\right).$$  For functions $f$ satisfying $0< f<1$ ($f$ is not necessarily the solution), we define  a functional  
\begin{equation}\label{Hch}\mathrm{H_{\chi}}[f]\colonequals \int_{\mathbb{T}^d}\int_{\mathbb{R}^d}S_{\chi}(f,v)dvdx.
\end{equation} 
If we have the existence of a solution $f\in (0,1)$ which is regular enough, one can check that, for all $t>0,$\begin{align}\label{dt H2}
\frac{d}{dt}\mathrm{H_{\chi}}[f(t)]=-\frac{1}{2}\int_{\mathbb{T}^d}\int_{\mathbb{R}^d}\int_{\mathbb{R}^d}\sigma(v,v')MM'(1-f)(1-f')(F-F')(\chi(F)-\chi(F'))dv'dvdx,
\end{align}
where $\displaystyle F\colonequals \frac{f}{M(1-f)}.$
Sine $\chi$ is an increasing function, we have $$(F-F')(\chi(F)-\chi(F'))\geq 0,$$ hence
$$\frac{d}{dt}\mathrm{H_{\chi}}[f(t)]\leq 0, \, \, \, \forall\, t>0.$$
 This means that, for any increasing function $\chi,$ the functional  $\mathrm{H_{\chi}}[f(t)]$ is an entropy for \eqref{Eq}.  If  $\chi$ is chosen strictly increasing, then the entropy dissipation functional  (the integral on the right hand side of \eqref{dt H2})  vanishes only if $\displaystyle \frac{f}{M(1-f)}$ does not depend on $v.$ This means 
\begin{equation}\label{loc}
f=\frac{\kappa(t,x) M(v)}{1+\kappa(t,x)M(v)}
\end{equation}
for some function $\kappa=\kappa(t,x).$ Hence, any function  $f$ in the form of \eqref{loc} is a \emph{local equilibrium} for \eqref{Eq}. For such $f $ the left hand side of \eqref{Eq} vanishes if $\kappa$ is constant. Hence, the unique  \emph{global equilibrium} for \eqref{Eq} is 
\begin{equation}\label{eqilib}
f_{\infty}(v)=\frac{\kappa_{\infty}M(v)}{1+\kappa_{\infty}M(v)},
\end{equation}
where $\kappa_{\infty}>0$ is determined by mass conservation $$\displaystyle \int_{\mathbb{T}^d}\int_{\mathbb{R}^d}f_{0}(x,v)dvdx=|\mathbb{T}^d|\int_{\mathbb{R}^d}f_{\infty}(v)dv.$$

On the basis of the decay of $\mathrm{H_{\chi}}[f(t)],$ one can conjecture that  $f(t)$ converges to the global equilibrium  $f_{\infty}$ as $t\to \infty.$   Thus, it is a basic problem to prove (or disprove) this convergence and to estimate the convergence rate. Our goal is to study the long time behavior of the  solution $f(t),$  and to prove that $f(t)$ converges exponentially to $f_{\infty}.$ We are interested in the study of rates of convergence and we want to derive constructive bounds for this convergence, because explicit and constructive estimates are essential for applications in physics (e.g. equilibration process, numerical simulations).

There are many works concerning the semiconductor Boltzmann equation \eqref{Eq}. For derivation and applications we refer the books \cite{Book1, Book2, Book3}. Existence and uniqueness of a global distributional solution to \eqref{Eq} was obtained in \cite{P.S, Neu.Schmeis}, where the authors used fixed point and maximum principle arguments.    \textcolor{black}{We mention existence and uniqueness results \cite{Exist1, Exist2, Pop., Exist3} for the semiconductor Boltzmann-Poisson equation.  There are many studies on the diffusion approximation  of \eqref{Eq} and   the asymptotics leads to a drift-diffusion equation, see \cite{Asym3, Asym5, Asym2, Asym4, Asym1, P.S}.}

Concerning the long time behavior of solutions, there are a few results. A polynomial rate of convergence to equilibrium has been obtained in \cite{Neu.Schmeis}. The proof based on the “entropy-entropy-dissipation” method \cite{D.V1, D.V, V} and uniform regularity bounds on the solution. Using the $H^1$ hypocoercivity method and assuming that the initial data $f_0$ is close to equilibrium, an exponential decay result was obtained in \cite{NeuMou}.

In this paper, we shall improve these previous results. We prove exponential decay of the solution  to global equilibrium based on the $L^2$-hypocoercivity approach. 
For spatially homogeneous kinetic equations it is well-known that explicit estimates on convergence to equilibrium can be obtained by the direct study of the relative entropy and its  dissipation functional. With the help of so-called Sobolev inequalities one can obtain a Gr\"onwall inequality for the relative entropy, which implies convergence of the solution to the global equilibrium. Unfortunately,  this idea can not be used for spatially inhomogeneous kinetic equations. The reason is that the collision operator acts only on the velocity variable  $v,$ which causes to loose information about $x-$direction in the dissipation functional. To deal with this problem, in  recent years, many new methods, so called hypocoercivity methods, are introduced to study the long-time behavior of spatially inhomogeneous kinetic equations.
The challenge of hypocoercivity is to understand the interplay between the collision operator that provides dissipativity in the velocity variable and the transport one which is conservative, in order to obtain global dissipativity for the whole problem. Largely motivated by the book of Villani \cite{V}, the literature on hypocoercivity has been growing considerably, with most of the approaches based on the
construction of suitable Lyapunov functionals. The $H^1$-based approach initiated in \cite{NeuMou} and expanded in \cite{V} is strongly motivated by
the theory of hypoellipticity. Recently it has been extended for certain model problems to prove sharp decay rates \cite{Achleitner2015,Arnold2014, AT1}.
The $L^2$ approach of \cite{DMS} has been strongly motivated by H\'erau \cite{Herau}. These methods have been applied to many different models, see e.g. \cite{AT2, Max, FPS, LP}

The main motivation of this work is to provide an extension of the hypocoercivity results for the model \eqref{Eq}. We prove exponential decay to equilibrium in $L^2$ without the close-to-equilibrium assumption on the initial data.
Motivated by the recent result \cite{FPS}, we construct a Lyapunov functional to \eqref{Eq} using the  $L^2$-hypocoercivity method \cite{DMS, FPS, BHR, Max}. Our goal is to show that this functional is equivalent to the square of the $L^2-$norm and satisfies a Gr\"onwal inequality, which implies  exponential decay of the solution to the global equilibrium. This functional depends on the projection of  the solution to the space of local equilibriums. In general, for semi-classical kinetic equations this projection operator  is nonlinear (see \eqref{pr} below). This makes our analysis challenging.

The results of this work provide an extension of the hypocoercivity results of \cite{ NeuMou, Neu.Schmeis}, since we do not require the close-to-equilibrium assumption on the initial data or
uniform regularity bounds on the solution. 


The organization of this paper is as follows. In Section 2, we state the main results. In Section 3, we construct a Lyapunov functional and establish some estimates  for the entropy dissipation, the collision operator, and the projection operator.  We  present the proof of our result  in Section 4. 

\section{Main results}
We first mention that global existence and uniqueness of a distributional solution follows from a combination of \cite[Theorem 2.1]{Neu.Schmeis} and \cite[Theorem 2.1]{P.S}: 
\begin{theorem}\label{existence}
\textcolor{black}{Assume there exist constants $\kappa_->0,$  $ \kappa_+>0,$ $\sigma_{-}>0,$ and $\sigma_{+}>0$ such that 
\begin{equation*}
    \sigma_{-}\leq \sigma(v,v')\leq \sigma_{+}, \, \, \forall\, v,\, v'\in \mathbb{R}^d
\end{equation*}
and 
\begin{equation*}
\frac{\kappa_{-} M(v)}{1+\kappa_{-} M(v)}\leq f_0(x,v)\leq \frac{\kappa_{+} M(v)}{1+\kappa_{+} M(v)}, \, \, \forall\, x \in \mathbb{T}^d, \,  \forall\, v\in \mathbb{R}^d.
\end{equation*}}
Then there is a unique distributional solution $f\in C([0,\infty), L^1(\mathbb{T}^d\times \mathbb{R}^d))$ to \eqref{Eq} satisfying the same bounds
\begin{equation}\label{bound}
\frac{\kappa_{-} M(v)}{1+\kappa_{-} M(v)}\leq f(t,x,v)\leq \frac{\kappa_{+} M(v)}{1+\kappa_{+} M(v)} 
\end{equation}
for all $t\geq 0,$ $x \in \mathbb{T}^d$ and $v\in \mathbb{R}^d.$ 
\end{theorem}
\begin{proof} 
       When $\sigma$ is constant, the proof follows by existence results in \cite[Theorem 2.1]{Neu.Schmeis} and \cite[Theorem 2.1]{P.S}.  This proof can be easily generalized if $\sigma$ is bounded from below and above by positive constants.
\end{proof}
We introduce a weighted $L^2-$space 
\begin{equation*}
L^2(\mathbb{R}^d, M)\colonequals \left\{g :\mathbb{R}^d \to \mathbb{R}: g \text{ is measurable and } \, \, \int_{\mathbb{T}^d}\int_{\mathbb{R}^d}\frac{g^2}{M}dvdx< \infty\right\}
\end{equation*}
with the scalar product
 \begin{equation*}
\langle g_1, g_2\rangle\colonequals \int_{\mathbb{T}^d}\int_{\mathbb{R}^d}\frac{g_1g_2}{M}dvdx, \, \, \, \, \, \, g_1,g_2 \in L^2(\mathbb{R}^d, M) 
\end{equation*}
and the induced norm $||\cdot||.$ Our main result is the following:
\begin{theorem}\label{main th}
\textcolor{black}{Assume there exist constants $\kappa_->0,$  $ \kappa_+>0,$ $\sigma_{-}>0,$ and $\sigma_{+}>0$ such that 
\begin{equation*}
    \sigma_{-}\leq \sigma(v,v')\leq \sigma_{+} 
\end{equation*}
\text{for all} $v,\, v'\in \mathbb{R}^d$ and 
\begin{equation*}
\frac{\kappa_{-} M(v)}{1+\kappa_{-} M(v)}\leq f_0(x,v)\leq \frac{\kappa_{+} M(v)}{1+\kappa_{+} M(v)}
\end{equation*} \text{for all } $x \in \mathbb{T}^d, \,   \, v\in \mathbb{R}^d.$ Then there exist  explicitly computable constants $\lambda, c>0$  such the solution $f$ to \eqref{Eq} satisfies
\begin{equation}\label{lamda}
||f(t)-f_{\infty}||\leq c e^{-\lambda t}||f_0-f_{\infty}||
\end{equation}
for all $t\geq 0.$} 
\end{theorem}

\section{Intermediate results }

We define a projection operator 
\begin{equation}\label{pr}
\Pi f\colonequals\frac{\kappa(t,x)M(v)}{1+\kappa(t,x)M(v)}, \, \, \, f\in L^2(\mathbb{R}^d, M),
\end{equation}
where $\kappa(t,x)$ is chosen by the condition
\begin{equation*}
\int_{\mathbb{R}^d}\Pi fdv=\int_{\mathbb{R}^d}f(t,x,v)dv.
\end{equation*}
$\Pi f$ means the projections of  $f$ to the space of functions which are local equilibrium for \eqref{Eq}. We note that   $ \Pi $  is a non-linear operator. For $v=(v_1,...,v_d)^T\in \mathbb{R}^d,$ we have 
\begin{equation}\label{vP=0}
\int_{\mathbb{R}^d}v\Pi f dv=\int_{\mathbb{R}^d}\frac{v\kappa(t,x)M(v)}{1+\kappa(t,x)M(v)}dv=0,
\end{equation}
and 
\begin{equation}\label{v_iv_jP}
\int_{\mathbb{R}^d}v_iv_j\Pi f dv=\int_{\mathbb{R}^d}\frac{v_i v_j \kappa(t,x)M(v)}{1+\kappa(t,x)M(v)}dv=\delta_{ij}\int_{\mathbb{R}^d}\frac{v_1^2\kappa(t,x)M(v)}{1+\kappa(t,x)M(v)}dv,\,\, i,j\in\{1,...,d\}, 
\end{equation} 
where  $\delta_{ij}=\begin{cases}1,
\,\,\, \text{if}\,\, i=j\\
0, \,\,\, \text{if}\,\, i\neq j
\end{cases}$ is the Dirac delta function.
We  denote 
$$\rho_{\infty}\colonequals \int_{\mathbb{R}^d}f_{\infty}(v)dv,$$
$$\rho(t,x)\colonequals \int_{\mathbb{R}^d}f(t,x,v)dv,$$
$$j(t,x)\colonequals \int_{\mathbb{R}^d}vfdv.$$
Let $\phi$ be the solution of the Poisson equation
\begin{equation*}
-\Delta_x \phi(t,x)=\rho(t,x)-\rho_{\infty},\, \, \, \int_{\mathbb{T}^d} \phi \,dx=0, \, \, \, x\in \mathbb{T}^d, \, \, t\geq 0.
\end{equation*}
Next we consider a specific entropy functional $\mathrm{H_{\chi}}$ (see \eqref{Hch}): we choose $\chi=\ln (z/{\kappa}_{\infty}),$ then  $\mathrm{H_{\chi}}$ equals
 \begin{equation*}
\mathrm{H}[f]\colonequals \int_{\mathbb{T}^d}\int_{\mathbb{R}^d} \left(f \ln{\frac{f}{f_{\infty}}}+(1-f)\ln{\frac{1-f}{1-f_{\infty}}}\right)dvdx.
\end{equation*}
 Then we define a modified entropy functional
\begin{equation*}
\mathrm{E}[f]\colonequals \mathrm{H}[f]+\delta \int_{\mathbb{T}^d} \nabla_x \phi\cdot j dx,
\end{equation*}
 where $\delta>0$ to be chosen later. 
\begin{lemma} Let the assumptions of Theorem \ref{existence} hold and $f$ be the solution to \eqref{Eq}.
\begin{itemize}
\item[(i)] There is a constant $c_1>0$ such that 
\begin{equation}\label{dt H<-c1}
\frac{d}{dt}\mathrm{H}[f(t)]\leq -c_1 ||f-\Pi f||^2.
\end{equation}
\item[(ii)] $Q$ is a bounded operator in $L^2(\mathbb{R}^d, M)$ and there is a constant $c_2>0$ such that
\begin{equation}\label{|Q|}
||Q(f)||\leq c_2||f-\Pi f||. 
\end{equation}

\item[(iii)] There are constant $c_3>0,$ $c_4>0$ and $c_5>0$ such that 
\begin{equation}\label{rk}
c_3\frac{(\Pi f-f_{\infty})^2}{M^2}\leq (\kappa-\kappa_{\infty})(\rho-\rho_{\infty})
\end{equation}
and 
\begin{equation}\label{rr}
c_4\frac{(\Pi f-f_{\infty})^2}{M^2}\leq (\rho-\rho_{\infty})^2\leq c_5\frac{(\Pi f-f_{\infty})^2}{M^2}.
\end{equation}
\item[(iv)] If $\delta>0$ is small enough, there are constants $c_6>0$ and $c_7>0$ such that \begin{equation}\label{E equiv}c_6||f-f_{\infty}||^2\leq \mathrm{E}[f]\leq c_7 ||f-f_{\infty}||^2.
\end{equation}
\end{itemize}
\end{lemma}
\begin{proof}
The proofs of $(i)$ and $(ii)$ are given in \cite[Lemma 3--4]{Neu.Schmeis}.

$(iii)$ By the definition of $\Pi$ we have  
$$\kappa-\kappa_{\infty}=\frac{\Pi f-f_{\infty}}{M(1-\Pi f)(1-f_{\infty})}$$
and $$\rho-\rho_{\infty}=\int_{\mathbb{R}^d}(\Pi f-f_{\infty})dv=(\kappa-\kappa_{\infty})\int_{\mathbb{R}^d}\frac{Mdv}{(1+\kappa M)(1+\kappa_{\infty}M)}.$$
These equations imply 
$$(\rho-\rho_{\infty})(\kappa-\kappa_{\infty})=\frac{(\Pi f-f_{\infty})^2}{M^2(1-\Pi f)^2(1-f_{\infty})^2}\int_{\mathbb{R}^d}\frac{Mdv}{(1+\kappa M)(1+\kappa_{\infty}M)}$$
and 
$$(\rho-\rho_{\infty})^2=\frac{(\Pi f-f_{\infty})^2}{M^2(1-\Pi f)^2(1-f_{\infty})^2}\left(\int_{\mathbb{R}^d}\frac{Mdv}{(1+\kappa M)(1+\kappa_{\infty}M)}\right)^2.$$
Since $1-\Pi f,$  $1-f_{\infty}$ and $\displaystyle \int_{\mathbb{R}^d}\frac{Mdv}{(1+\kappa M)(1+\kappa_{\infty}M)}$ are bounded from below and above by positive constants,  the claimed inequalities hold.

$(iv)$  We estimate  
$$ \left|\int_{\mathbb{T}^d} \nabla_x \phi\cdot j dx\right|\leq ||\nabla_x \phi||_{L^2(\mathbb{T}^d)}||j||_{L^2(\mathbb{T}^d)}.$$
Because of the elliptic regularity $||\nabla_x \phi||_{L^2(\mathbb{T}^d)}$ is bounded (up to a constant) by $||\rho-\rho_{\infty}||_{L^2(\mathbb{T}^d)}.$ Moreover, the H\"older inequality provides $||\rho-\rho_{\infty}||_{L^2(\mathbb{T}^d)}\leq ||f-f_{\infty}||.$  
We use $\displaystyle \int_{\mathbb{R}^d }vf_{\infty}dv=0$ and the H\"older inequality 
\begin{align*}
||j(t)||_{L^2(\mathbb{T}^d)}&=\sqrt{\int_{\mathbb{T}^d}\left|\int_{\mathbb{R}^d}v(f-f_{\infty})dv \right|^2dx}\\
&\leq \sqrt{\int_{\mathbb{T}^d}\left(\int_{\mathbb{R}^d}|v|^2M(v)dv \right)\left(\int_{\mathbb{R}^d}\frac{(f-f_{\infty})^2}{M}dv\right) dx}\\
&\leq \sqrt{d}||f-f_{\infty}||. 
\end{align*}
These estimates show that there is a constant $c_8$ such that
$$ \left|\int_{\mathbb{T}^d} \nabla_x \phi\cdot j dx\right|\leq c_8||f-f_{\infty}||^2.$$
Therefore, we have
$$\mathrm{H}[f]-\delta c_8||f-f_{\infty}||^2\leq \mathrm{E}[f]\leq \mathrm{H}[f]+\delta c_8||f-f_{\infty}||^2.$$
It is proven in \cite[Lemma 1]{Neu.Schmeis} that $\mathrm{H}[f]$ is equivalent to $||f-f_{\infty}||^2,$ when $f$ satisfies the bounds \eqref{bound}.  Hence, if $\delta$ is small enough, $\mathrm{E}[f]$ is equivalent to $||f-f_{\infty}||^2.$
\end{proof}
 
\section{Proof of Theorem \ref{main th}}

\begin{proof}[Proof of Theorem \ref{main th}]
We consider the modified entropy functional \begin{equation*}
\mathrm{E}[f(t)]\colonequals \mathrm{H}[f(t)]+\delta \int_{\mathbb{T}^d} \nabla_x \phi\cdot j dx.
\end{equation*}
The time derivative of this functional equals
\begin{align}\label{dt E}
\frac{d}{dt}\mathrm{E}[f(t)]=\frac{d}{dt}\mathrm{H}[f(t)]+\delta \int_{\mathbb{T}^d} \partial_t\nabla_x \phi\cdot j dx+\delta \int_{\mathbb{T}^d} \nabla_x \phi\cdot \partial_t j dx
\end{align}
We estimate this time derivative in the following three steps.\\
\textbf{Step 1: Estimates on $\displaystyle \int_{\mathbb{T}^d} \partial_t\nabla_x \phi\cdot j dx.$ }\\
 By integrating \eqref{Eq} with respect to $v$ and using $\displaystyle \int_{\mathbb{R}^d} Q(v)dv=0,$ we obtain
\begin{equation*}
\partial_t \rho+\text{div}_xj=0.
\end{equation*}
This shows that $-\partial_t \Delta \phi=-\text{div}_xj.$ Integrating by parts and using the H\"older inequality, we estimate 
$$||\partial_t \nabla_x \phi(t)||_{L^2(\mathbb{T}^d)}^2=-\int_{\mathbb{T}^d}\partial_t  \phi ~\text{div}_xjdx=\int_{\mathbb{T}^d}\partial_t  \nabla_x \phi \cdot jdx\leq ||j(t)||_{L^2(\mathbb{T}^d)}||\partial_t \nabla_x \phi(t)||_{L^2(\mathbb{T}^d)}. $$
This shows  
$$||\partial_t \nabla_x \phi(t)||_{L^2(\mathbb{T}^d)}\leq ||j(t)||_{L^2(\mathbb{T}^d)}. $$ 
 We use \eqref{vP=0} and the H\"older inequality 
$$||j(t)||^2_{L^2(\mathbb{T}^d)}=\int_{\mathbb{T}^d}\left|\int_{\mathbb{R}^d}v(f-\Pi f)dv \right|^2dx\leq d\int_{\mathbb{T}^d}\int_{\mathbb{R}^d}\frac{(f-\Pi f)^2}{M}dv dx. $$
The H\"older inequality and the last two estimates show 
\begin{equation}\label{tphi j}\int_{\mathbb{T}^d} \partial_t\nabla_x \phi\cdot j dx\leq d ||f-\Pi f||^2.
\end{equation}\\
\textbf{Step 2: Estimates on $\displaystyle \int_{\mathbb{T}^d} \nabla_x \phi\cdot \partial_t j dx.$}\\
 Multiplying \eqref{Eq} by $v,$  then integrating with respect to $v,$ we obtain
\begin{align*}
\partial_t j=-\int_{\mathbb{R}^d}v(v\cdot \nabla_x f)dv+\int_{\mathbb{R}^d}vQ(f)dv.
\end{align*}
Now, we use \eqref{v_iv_jP} to write 
 \begin{align*}
 \partial_t j&=-\int_{\mathbb{R}^d}v(v\cdot \nabla_x (\Pi f-f_{\infty}))dv-\int_{\mathbb{R}^d}v(v\cdot \nabla_x (f-\Pi f))dv+\int_{\mathbb{R}^d}vQ(f)dv\\
 &=-\nabla_x\left(\int_{\mathbb{R}^d}v_1^2  (\Pi f-f_{\infty})dv\right)-\int_{\mathbb{R}^d}v(v\cdot \nabla_x (f-\Pi f))dv+\int_{\mathbb{R}^d}vQ(f)dv.
 \end{align*}
 We use this equation to compute 
 \begin{align}\label{tj}
 \int_{\mathbb{T}^d} \nabla_x \phi\cdot \partial_t j dx=&-\int_{\mathbb{T}^d} \nabla_x \phi\cdot \nabla_x\left(\int_{\mathbb{R}^d}v_1^2  (\Pi f-f_{\infty})dv \right)dx\nonumber\\
 &-\int_{\mathbb{T}^d}\int_{\mathbb{R}^d}(\nabla_x \phi\cdot v)(v\cdot \nabla_x (f-\Pi f))dvdx\nonumber\\
 &+\int_{\mathbb{T}^d}\int_{\mathbb{R}^d}(\nabla_x \phi\cdot v)Q(f)dvdx.
 \end{align}
 We now estimate each term in this equation. First we have 
 \begin{multline*}
 -\int_{\mathbb{T}^d} \nabla_x \phi\cdot \nabla_x\left(\int_{\mathbb{R}^d}v_1^2  (\Pi f-f_{\infty})dv \right)dx\\=-\int_{\mathbb{T}^d}(\rho-\rho_{\infty})(\kappa-\kappa_{\infty})\left(\int_{\mathbb{R}^d} \frac{v_1^2M(v)dv}{(1+\kappa M(v))(1+\kappa_{\infty}M(v)) }  dv\right)dx. 
 \end{multline*}
 Since the integral $\displaystyle \int_{\mathbb{R}^d} \frac{v_1^2M(v)dv}{(1+\kappa M(v))(1+\kappa_{\infty}M(v)) }  dv$ is bounded from below by a positive constant, \eqref{rk} implies 
 \begin{equation}\label{tj1}
 -\int_{\mathbb{T}^d} \nabla_x \phi\cdot \nabla_x\left(\int_{\mathbb{R}^d}v_1^2  (\Pi f-f_{\infty})dv \right)dx\leq -c_9||\Pi f-f_{\infty}||^2
 \end{equation}
 for some constant $c_9>0.$
 For the second term in \eqref{tj} we integrate by parts and use the H\"older inequality
 \begin{align*}
 -\int_{\mathbb{T}^d}\int_{\mathbb{R}^d}(\nabla_x \phi\cdot v)(v\cdot \nabla_x (f-\Pi f))dvdx&=\int_{\mathbb{T}^d}\int_{\mathbb{R}^d} v^T\frac{\partial^2 \phi}{\partial x^2} v (f-\Pi f)dvdx\\
 & \leq ||f-\Pi f||\sqrt{\int_{\mathbb{T}^d}\int_{\mathbb{R}^d} |v|^2\left|\frac{\partial^2 \phi}{\partial x^2}\right|^2 Mdvdx}\\
 &=\sqrt{d}||f-\Pi f||\left|\left|\frac{\partial^2 \phi}{\partial x^2}\right|\right|_{L^2(\mathbb{T}^d)}.
 \end{align*}
 $\left|\left|\frac{\partial^2 \phi}{\partial x^2}\right|\right|_{L^2(\mathbb{T}^d)}$ is bounded (up to a constant) by $||\rho-\rho_{\infty}||_{L^2(\mathbb{R}^d)}$ because of the elliptic regularity. Moreover, \eqref{rr} implies that   $||\rho-\rho_{\infty}||_{L^2(\mathbb{T}^d)}$ is bounded (up to a constant) by $||\Pi f-f_{\infty}||.$ Hence, there is a constant $c_{10}>0$ such that 
 \begin{equation}\label{tj2}
  -\int_{\mathbb{T}^d}\int_{\mathbb{R}^d}(\nabla_x \phi\cdot v)(v\cdot \nabla_x (f-\Pi f))dvdx\leq c_{10} || f-\Pi f||\,||\Pi f-f_{\infty}||.
 \end{equation}
 We estimate the last term in \eqref{tj} by using the H\"older inequality and \eqref{|Q|}
 \begin{align*}
 \int_{\mathbb{T}^d}\int_{\mathbb{R}^d}(\nabla_x \phi\cdot v)Q(f)dvdx & \leq ||Q||\sqrt{\int_{\mathbb{T}^d}\int_{\mathbb{R}^d}|\nabla_x \phi|^2|v|^2M(v)dvdx}\\
 &\leq \sqrt{d}c_2||f-\Pi f||\,||\nabla_x \phi||_{L^2(\mathbb{T}^d)}.
 \end{align*}
 Similarly, $||\nabla_x \phi||_{L^2(\mathbb{T}^d)}$  is bounded (up to a constant) by $||\rho-\rho_{\infty}||_{L^2(\mathbb{R}^d)}$ because of the elliptic regularity. Moreover, \eqref{rr} implies that   $||\rho-\rho_{\infty}||_{L^2(\mathbb{T}^d)}$ is bounded (up to a constant) by $||\Pi f-f_{\infty}||.$ Hence, there is a constant $c_{11}>0$ such that 
 \begin{equation}\label{tj3}
 \int_{\mathbb{T}^d}\int_{\mathbb{R}^d}(\nabla_x \phi\cdot v)Q(f)dvdx  \leq c_{11}|| f-\Pi f||\,||\Pi f-f_{\infty}||.
 \end{equation}
 \eqref{tj}, \eqref{tj1}, \eqref{tj2} and \eqref{tj3} imply 
 \begin{equation}\label{phitj}\int_{\mathbb{T}^d} \nabla_x \phi\cdot \partial_t j dx\leq -c_9||\Pi f-f_{\infty}||^2+(c_{10}+c_{11}) || f-\Pi f||\,||\Pi f-f_{\infty}||.
 \end{equation}\\
 \textbf{Step 3: Gr\"onwall's inequality. }\\
 We now summarize all estimates. \eqref{dt E}, \eqref{dt H<-c1}, \eqref{tphi j} and \eqref{phitj} yield
 $$\frac{d}{dt}\mathrm{E}[f(t)]\leq -(c_1-\delta d)||f-\Pi f||^2-\delta c_9||\Pi f-f_{\infty}||^2+\delta(c_{10}+c_{11}) || f-\Pi f||\,||\Pi f-f_{\infty}||.$$
 One can check that if $\delta>0$ is small enough, there is a constant $c_{12}=c_{12}(\delta)>0$ such that 
 $$\frac{d}{dt}\mathrm{E}[f(t)]\leq -c_{12}(||f-\Pi f||^2+||\Pi f-f_{\infty}||^2)\leq  -\frac{c_{12}}{2}||f-f_{\infty}||^2.$$
 The first inequality in \eqref{E equiv} implies
 $$\frac{d}{dt}\mathrm{E}[f(t)]\leq-\frac{c_{12}}{2c_6}\mathrm{E}[f(t)].$$
 Then, Gr\"onwall's inequality implies 
 $$\mathrm{E}[f(t)]\leq e^{-\frac{c_{12}}{2c_6}t}\mathrm{E}[f_0], \, \, \, \forall\, t\geq 0.$$
 This inequality and \eqref{E equiv} provide \eqref{lamda} with $c\colonequals \sqrt{\frac{c_7}{c_6}}$ and $\lambda\colonequals \frac{c_{12}}{4c_6}.$
\end{proof}

 \section{Acknowledgements:}
 Marlies Pirner and Gayrat Toshpulatov was funded by the Deutsche Forschungsgemeinschaft (DFG, German Research Foundation) under Germany’s Excellence Strategy EXC 2044-390685587, Mathematics Münster: Dynamics–Geometry–Structure. In addition, Marlies Pirner was funded by the German Science Foundation DFG (grant no. PI 1501/2-1).


{}

\end{document}